\documentclass{amsart}

\usepackage{mathrsfs}
\usepackage{mathrsfs}
\usepackage{amsfonts}
\usepackage{amsfonts}
\usepackage{mathrsfs}
\usepackage{amsfonts}
\usepackage{amsfonts}
\usepackage{amsfonts}
\usepackage{mathrsfs}
\usepackage{amsfonts}
\usepackage{amssymb,amsmath}
\usepackage{amsthm}
\usepackage{amsfonts}

\newtheorem{theorem}{Theorem}[section]
\newtheorem{lemma}[theorem]{Lemma}

\theoremstyle{definition}
\newtheorem{definition}[theorem]{Definition}

\theoremstyle{remark}

\numberwithin{equation}{section}

\begin{document}

\title[]
 {Existence of non-trivial solutions for nonlinear fractional Schr\"{o}dinger-Poisson equations}


\author[]{Kexue Li}
\address{School of Mathematics and Statistics,
Xi'an Jiaotong University,
 Xi'an
710049, China.}
\email{kexueli@gmail.com}

\subjclass[2010]{35Q40, 58E30}



\keywords{Fractional Schr\"{o}dinger-Poisson equation; Perturbation method; Nontrivial solution.}

\begin{abstract}
We study the nonlinear fractional Schr\"{o}dinger-Poisson equations
\begin{equation*}
 \ \left\{\begin{aligned}
&(-\Delta)^{s}u+u+\phi u=f(x,u), \ \ \ \ \mbox{in} \ \mathbb{R}^{3},\\
&(-\Delta)^{t}\phi=u^{2}, \ \ \ \  \ \  \ \ \ \  \ \ \ \  \ \ \ \ \ \   \  \ \  \mbox{in} \ \mathbb{R}^{3},
\end{aligned}\right.
\end{equation*}
where $s,t\in (0,1]$, $2t+4s>3$. Under some assumptions on $f$, we obtain the existence of non-trivial solutions. The proof is based on the perturbation method and the mountain pass theorem.

\end{abstract}

\maketitle


\section{Introduction}
In this paper, we are concerned with the existence of non-trivial solutions for the following fractional Schr$\ddot{\mbox{o}}$dinger-Poisson equation
\begin{equation}\label{fsp}
 \ \left\{\begin{aligned}
&(-\Delta)^{s}u+u+\phi u=f(x,u), \ \ \ \ \mbox{in} \ \mathbb{R}^{3},\\
&(-\Delta)^{t}\phi=u^{2}, \ \ \ \  \ \  \ \ \ \  \ \ \ \  \ \ \ \ \ \   \  \ \  \mbox{in} \ \mathbb{R}^{3},
\end{aligned}\right.
\end{equation}
where $s,t\in (0,1]$, $2t+4s>3$, $(-\Delta)^{s}$ denotes the fractional Laplacian.

When $s=t=1$, the equation (\ref{fsp}) reduces to Schr$\ddot{\mbox{o}}$dinger-Poisson equation, which describes system of identical charged particles interacting each other in the case where magnetic effects can be neglected \cite{V,Ruiz}.
When $\phi=0$, (\ref{fsp}) reduces to a fractional Schr$\ddot{\mbox{o}}$dinger equation, which is a fundamental equation in fractional quantum mechanics \cite{Laskin,N}.

Recently, some authors proposed a new approach called perturbation method to study the quasilinear elliptic equations, see \cite{LLW}. The idea is to get the existence of critical points of the perturbed energy functional $I_{\lambda}$ for $\lambda>0$ small and then taking $\lambda\rightarrow 0$ to obtain solutions of original problems. Very recently, Feng \cite{Feng} used the perturbation method to study the Schr$\ddot{\mbox{o}}$dinger-Poisson equation
\begin{equation}\label{Feng}
 \ \left\{\begin{aligned}
&-\Delta u+u+\phi u=f(x,u), \ \ \ \ \ \  \  \ \ \ \mbox{in} \ \mathbb{R}^{3},\\
&(-\Delta)^{\alpha/2}\phi=u^{2}, \ \lim_{|x|\rightarrow \infty}\phi(x)=0, \ \ \mbox{in} \ \mathbb{R}^{3},
\end{aligned}\right.
\end{equation}
where $\alpha\in (1,2]$. Under some conditions, the problem (\ref{Feng}) possesses at least a nontrivial solution.

We point out that when  $s=1$ and $t\in (\frac{1}{2},1]$, the problem (\ref{fsp}) boils down to (\ref{Feng}). The main result of this paper is described as follows.
\begin{theorem}\label{theorem}
Suppose $f$ satisfies the following conditions: \\
(A1) For every $x\in \mathbb{R}^{3}$ and $u\in \mathbb{R}$, there exist constants $C_{1}>0$ and $p\in [2,2^{\ast}_{s})$ such that
\begin{align*}
|f(x,u)|\leq C_{1}(|u|+|u|^{p-1}),
\end{align*}
where $2^{\ast}_{s}=\frac{6}{3-2s}$ is the fractional critical Sobolev exponent; \\
(A2) $f(x,u)=o(|u|)$, $|u|\rightarrow 0$, uniformly on $\mathbb{R}^{3}$; \\
(A3) there exists $\mu>4$ such that
\begin{align*}
0<\mu F(x,u)\leq uf(x,u)
\end{align*}
holds for every $x\in \mathbb{R}^{3}$ and $u\in \mathbb{R}\backslash \{0\}$, where $F(x,u)=\int_{0}^{u}f(x,s)ds$; \\
Then problem (\ref{fsp}) has at least a nontrivial solution.
\end{theorem}

The paper is organized as follows. In Section 2, we present some preliminaries results.
In Section 3, we will prove Theorem \ref{theorem}.

\section{Preliminaries}
For $p\in [1,\infty)$, we denote by $L^{p}(\mathbb{R}^{3})$ the usual Lebesgue space with the norm $\|u\|_{p}=\left(\int_{\mathbb{R}^{3}}|u|^{p}dx\right)^{\frac{1}{p}}$. For any $p\in [1,\infty)$ and $s\in (0,1)$, we recall some definitions of fractional Sobolev spaces and the fractional Laplacian $(-\Delta)^{s}$, for more details, we refer to \cite{DPV}.  $H^{s}(\mathbb{R}^{3})$ is defined as follows
\begin{align*}
H^{s}(\mathbb{R}^{3})=\left\{u\in L^{2}(\mathbb{R}^{3}): \int_{\mathbb{R}^{3}}(1+|\xi|^{2s})|\mathcal{F}u(\xi)|^{2}d\xi<\infty\right\}
\end{align*}
with the norm
\begin{align}\label{alternative}
\|u\|_{H^{s}}=\big(|\mathcal{F}u(\xi)|^{2}+|\xi|^{2s}|\mathcal{F}u(\xi)|^{2}\big)^{\frac{1}{2}},
\end{align}
where $\mathcal{F}u$ denotes the Fourier transform of $u$. By $\mathcal{S}(\mathbb{R}^{n})$, we denote the Schwartz space of rapidly decaying $C^{\infty}$ functions in $\mathbb{R}^{n}$. For $u\in \mathcal{S}(\mathbb{R}^{n})$ and $s\in (0,1)$, $(-\Delta)^{s}$ is defined by
\begin{align*}
(-\Delta)^{s}f=\mathcal{F}^{-1}(|\xi|^{2s}(\mathcal{F}f)), \ \forall \xi\in \mathbb{R}^{n}.
\end{align*}
By Plancherel's theorem, we have $\|\mathcal{F}u\|_{2}=\|u\|_{2}$, $\||\xi|^{s}\mathcal{F}u\|_{2}=\|(-\Delta)^{\frac{s}{2}}u\|$. Then by (\ref{alternative}), we get the equivalent norm
\begin{align*}
\|u\|_{H^{s}}=\left(\int_{\mathbb{R}^{3}}(|(-\Delta)^{\frac{s}{2}}u(x)|^{2}+|u(x)|^{2})dx\right)^{\frac{1}{2}}.
\end{align*}
For $s\in (0,1)$, the fractional Sobolev space $D^{s,2}(\mathbb{R}^{3})$ is defined as follows
\begin{align*}
D^{s,2}(\mathbb{R}^{3})=\left\{u\in L^{2^{\ast}_{s}}(\mathbb{R}^{3}): |\xi|^{s}\mathcal{F}u(\xi)\in L^{2}(\mathbb{R}^{3})\right\},
\end{align*}
which is the completion of $C^{\infty}_{0}(\mathbb{R}^{3})$ with respect to the norm
\begin{align*}
\|u\|_{D^{s,2}}=\left(\int_{\mathbb{R}^{3}}|(-\Delta)^{\frac{s}{2}}u|^{2}dx\right)^{\frac{1}{2}}=\left(\int_{\mathbb{R}^{3}}|\xi|^{2s}|\mathcal{F}u(\xi)|^{2}d\xi\right)^{\frac{1}{2}}.
\end{align*}

\begin{lemma}\label{best}\emph{(Theorem 2.1 in \cite{C}).}
For any $s\in (0,\frac{3}{2})$, $D^{s,2}(\mathbb{R}^{3})$ is continuously embedded in $L^{2^{\ast}_{s}}(\mathbb{R}^{3})$, i.e., there exists $c_{s}>0$ such that
\begin{align*}
\left(\int_{\mathbb{R}^{3}}|u|^{2^{\ast}_{s}}dx\right)^{2/2^{\ast}_{s}}\leq c_{s}\int_{\mathbb{R}^{3}}|(-\Delta)^{\frac{s}{2}}u|^{2}dx, \ u\in D^{s,2}(\mathbb{R}^{3}).
\end{align*}
\end{lemma}
We consider the variational setting of (\ref{fsp}). From Theorem 6.5 and Corollary 7.2 in \cite{DPV}, it is known that the space $H^{s}(\mathbb{R}^{3})$  is continuously embedded in $L^{q}(\mathbb{R}^{3})$ for any $q\in [2,2^{\ast}_{s}]$
 and the embedding $H^{s}(\mathbb{R}^{3})\hookrightarrow L^{q}(\mathbb{R}^{3})$ is locally compact for $q\in [1, 2^{\ast}_{s})$. \\
If $2t+4s> 3$, then $H^{s}(\mathbb{R}^{3})\hookrightarrow L^{\frac{12}{3+2t}}(\mathbb{R}^{3})$. For $u\in H^{s}(\mathbb{R}^{3})$, the linear operator $T_{u}: D^{t,2}(\mathbb{R}^{3})\rightarrow \mathbb{R}$ defined as
\begin{align*}
T_{u}(v)=\int_{\mathbb{R}^{3}}u^{2}vdx,
\end{align*}
By H$\ddot{\mbox{o}}$lder inequality and Lemma \ref{best},
\begin{align}\label{dt1}
|T_{u}(v)|\leq \|u\|^{2}_{12/(3+2t)}\|v\|_{2^{\ast}_{t}}\leq C\|u\|^{2}_{H^{s}}\|v\|_{D^{t,2}}.
\end{align}
Set
\begin{align*}
\eta(u,v)=\int_{\mathbb{R}^{3}}(-\Delta)^{\frac{t}{2}}u\cdot (-\Delta)^{\frac{t}{2}}vdx, \ u,v\in D^{t,2}(\mathbb{R}^{3}).
\end{align*}
It is clear that $\eta(u,v)$ is bilinear, bounded and coercive. The Lax-Milgram theorem implies that for every $u\in H^{s}(\mathbb{R}^{3})$, there exists a unique $\phi^{t}_{u}\in D^{t,2}(\mathbb{R}^{3})$ such that $T_{u}(v)=\eta(\phi_{u},v)$ for any $v\in D^{t,2}(\mathbb{R}^{3})$, that is
\begin{align}\label{dt2}
\int_{\mathbb{R}^{3}}(-\Delta)^{\frac{t}{2}}\phi^{t}_{u}(-\Delta)^{\frac{t}{2}}vdx=\int_{\mathbb{R}^{3}}u^{2}vdx.
\end{align}
Therefore, $(-\Delta)^{t}\phi^{t}_{u}=u^{2}$ in a weak sense. Moreover,
\begin{align}\label{LM}
\|\phi^{t}_{u}\|_{D^{t,2}}=\|T_{u}\|\leq C\|u\|^{2}_{H^{s}}.
\end{align}
Since $t\in (0,1]$ and $2t+4s>3$, then $\frac{12}{3+2t}\in (2,2^{\ast}_{s})$.
From Lemma \ref{best}, (\ref{dt1}) and (\ref{dt2}), it follows that
\begin{align}\label{dt}
\|\phi^{t}_{u}\|^{2}_{D^{t,2}}&=\int_{\mathbb{R}^{3}}|(-\Delta)^{\frac{t}{2}}\phi^{t}_{u}|^{2}dx=\int_{\mathbb{R}^{3}}u^{2}\phi^{t}_{u}dx\leq \|u\|^{2}_{\frac{12}{3+2t}}\|\phi^{t}_{u}\|_{2^{\ast}_{t}}\leq C\|u\|^{2}_{\frac{12}{3+2t}}\|\phi^{t}_{u}\|_{D^{t,2}}.
\end{align}
Then
\begin{align}\label{phi}
\|\phi^{t}_{u}\|_{D^{t,2}}\leq C\|u\|^{2}_{\frac{12}{3+2t}}.
\end{align}
For $x\in \mathbb{R}^{3}$, we have
\begin{align}\label{Riesz}
\phi^{t}_{u}(x)=c_{t}\int_{\mathbb{R}^{3}}\frac{u^{2}(y)}{|x-y|^{3-2t}}dy,
\end{align}
which is the Riesz potential \cite{Stein}, where
\begin{align*}
c_{t}=\frac{\Gamma(\frac{3-2t}{2})}{\pi^{3/2}2^{2t}\Gamma(t)}.
\end{align*}
Substituting $\phi^{t}_{u}$ in (\ref{fsp}), we have the fractional Schr$\ddot{\mbox{o}}$dinger equation
\begin{align}\label{fractional Schrodinger}
(-\Delta)^{s}u+u+\phi^{t}_{u}u=f(x,u), \ x\in \mathbb{R}^{3},
\end{align}
The energy functional $I: H^{s}(\mathbb{R}^{3})\rightarrow \mathbb{R}$ corresponding to problem (\ref{fractional Schrodinger}) is defined by
\begin{align*}
I(u)=\frac{1}{2}\int_{\mathbb{R}^{3}}(|(-\Delta)^{\frac{s}{2}}u|^{2}+u^{2})dx+\frac{1}{4}\int_{\mathbb{R}^{3}}\phi^{t}_{u}u^{2}dx-\int_{\mathbb{R}^{3}}F(x,u)dx.
\end{align*}
It is easy to see that $I$ is well defined in $H^{s}(\mathbb{R}^{3})$ and $I\in C^{1}(H^{s}(\mathbb{R}^{3}),\mathbb{R})$, and
\begin{align}\label{langle}
\langle I'(u),v\rangle=\int_{\mathbb{R}^{3}}\left((-\Delta)^{\frac{s}{2}}u(-\Delta)^{\frac{s}{2}}v+uv+\phi^{t}_{u}uv-f(x,u)v\right)dx, \ \ v\in H^{s}(\mathbb{R}^{3}).
\end{align}

\begin{definition}
(1) We call $(u,\phi)\in H^{s}(\mathbb{R}^{3})\times D^{t,2}(\mathbb{R}^{3})$ is a weak solution of (\ref{fsp}) if $u$ is a weak solution of (\ref{fractional Schrodinger}).\\
(2)We call  $u$ is a weak solution of (\ref{fractional Schrodinger}) if
\begin{align*}
\int_{\mathbb{R}^{3}}\left((-\Delta)^{\frac{s}{2}}u(-\Delta)^{\frac{s}{2}}v+uv+\phi^{t}_{u}uv-f(x,u)v\right)dx=0,
\end{align*}
for any $v\in H^{s}(\mathbb{R}^{3})$.
\end{definition}

Assume that the potential $V(x)$ satisfies the condition\\
(V)  $V\in C(\mathbb{R}^{3})$, $inf_{x\in R^{3}}V(x)\geq V_{0}>0$, where $V_{0}$ is a constant.
For every $M>0$, $meas\{x\in \mathbb{R}^{3}: V(x)\leq M\}<\infty$, where $meas(\cdot)$ denotes the Lebesgue measure in $\mathbb{R}^{3}$.

Let
\begin{align*}
E=\left\{u\in H^{s}(\mathbb{R}^{3}): \int_{\mathbb{R}^{3}}(|(-\Delta)^{\frac{s}{2}}u|^{2}+V(x)u^{2})dx<\infty\right\}.
\end{align*}
Then $E$ is a Hilbert space with the inner product
\begin{align*}
\langle u,v\rangle_{E}=\int_{\mathbb{R}^{3}}\big((-\Delta)^{\frac{s}{2}}u(-\Delta)^{\frac{s}{2}}v+V(x)uv\big)dx
\end{align*}
and the norm $\|u\|_{E}=\langle u,u\rangle^{\frac{1}{2}}_{E}$. By Lemma 2.3 in \cite{TK}, it is known that $E$ is compactly embedded in $L^{p}(\mathbb{R}^{3})$ for $2\leq p<2^{\ast}_{s}$.

For fixed $\lambda\in (0,1]$, we introduce the following inner product
\begin{align*}
\langle u,v\rangle_{H^{s}_{\lambda}}=\int_{\mathbb{R}^{3}}\big((-\Delta)^{\frac{s}{2}}u(-\Delta)^{\frac{s}{2}}v+\lambda V(x)uv\big)dx
\end{align*}
and the norm $\|u\|_{H^{s}_{\lambda}}=\langle u,u\rangle^{\frac{1}{2}}_{H^{s}_{\lambda}}$. Denote $E_{\lambda}=(E,\|\cdot\|_{H^{s}_{\lambda}})$.

Define the perturbed functional $I_{\lambda}: E\rightarrow \mathbb{R}$:
\begin{align}\label{functional}
I_{\lambda}(u)=I(u)+\frac{\lambda}{2}\int_{\mathbb{R}^{3}}V(x)u^{2}dx, \ \ \lambda\in (0,1].
\end{align}
\begin{lemma}\label{geometry}
Suppose that $V(x)\geq 0$ and $(A1)$, $(A2)$ hold. Then there exist $\rho>0$, $\eta>0$ such that for fixed $\lambda\in (0,1]$,
\begin{align*}
\inf_{u\in E, \ \|u\|_{E}=\rho}I_{\lambda}(u)>\eta,
\end{align*}
where $\rho$ and $\eta$ are independent of $\lambda$.
\end{lemma}
\begin{proof}
By $(A1)$ and $(A2)$, for any $\varepsilon>0$, there exists $C_{\varepsilon}>0$ such that
\begin{align*}
|f(x,u)|\leq \varepsilon |u|+C_{\varepsilon}|u|^{p-1}, x \in\mathbb{R}^{3}, \ u\in \mathbb{R}.
\end{align*}
Then
\begin{align*}
|F(x,u)|\leq \frac{\varepsilon}{2}|u|^{2}+\frac{C_{\varepsilon}}{p}|u|^{p}.
\end{align*}
For $\rho>0$, set
\begin{align*}
\Sigma_{\rho}=\{u\in E: \|u\|_{E}\leq \rho\}.
\end{align*}
It is known that $E$ is continuously embedded into $L^{q}(\mathbb{R}^{3})$ for $q\in [2,2_{s}^{\ast}]$ $(2_{s}^{\ast}=\frac{6}{3-2s})$, then $\|u\|_{q}\leq C_{0}\|u\|_{E}$.  Since $p\in (2,2^{\ast}_{s})$, for $u\in \partial \Sigma_{\rho}$,
\begin{align*}
I_{\lambda}(u)&=\frac{1}{2}\|u\|^{2}_{E}+\frac{\lambda}{2}\int_{\mathbb{R}^{3}}V(x)u^{2}dx+\frac{1}{4}\int_{\mathbb{R}^{3}}\phi^{t}_{u}(x)u^{2}dx
-\int_{\mathbb{R}^{3}}F(x,u)dx\\
&\geq \frac{(1-\varepsilon)\rho^{2}}{2}-\frac{C_{\varepsilon}C_{0}}{p}\rho^{p}.
\end{align*}
For $\varepsilon\in (0,1)$ and sufficiently small $\rho$, the conclusion holds.
\end{proof}

\begin{lemma}\label{lambda}
Assume that $(A3)$ and $(A4)$ hold. Then there exists $e\in E$ with $\|e\|_{E}>\rho$ such that $I_{\lambda}(e)<0$ for fixed $\lambda\in(0,1]$, where $\rho$ is the same as in Lemma \ref{geometry}.
\begin{proof}
By $(A3)$, there exists a constant $C>0$ such that
\begin{align}\label{lower bound}
F(x,u)\geq C|u|^{\mu}, \ u\in \mathbb{R}.
\end{align}
By (\ref{LM}), (\ref{dt}),
\begin{align}\label{LMdt}
\int_{\mathbb{R}^{3}}\phi^{t}_{u}u^{2}dx=\|\phi^{t}_{u}\|^{2}_{D^{t,2}}\leq C\|u\|^{4}_{H^{s}}.
\end{align}
For $\xi> 0$ and $v\in C_{0}^{\infty}(\mathbb{R}^{3})$, by (\ref{functional}), (\ref{lower bound}) and (\ref{LMdt}), we have
\begin{align*}
I_{\lambda}(\xi v)&=\frac{\xi^{2}}{2}\|v\|^{2}_{H^{s}_{\lambda}}+\frac{\xi^{2}}{2}\|v\|_{2}^{2}+\frac{\xi^{4}}{4}\int_{R^{3}}\phi_{v}^{t}v^{2}dx-\int_{R^{3}}F(x,\xi v)dx\\
&\leq \frac{\xi^{2}}{2}\|v\|^{2}_{E}+\frac{\xi^{2}}{2}\|v\|_{2}^{2}+\frac{C\xi^{4}}{4}\|v\|^{4}_{H^{s}}-C\xi^{\mu}\|v\|^{\mu}_{\mu}\rightarrow -\infty
\end{align*}
as $\xi \rightarrow +\infty$.
Define a path $h: [0,1]\rightarrow E$ by $h(\eta)=\eta \xi v$. For $\xi$ large enough, we get
\begin{align*}
\|h(1)\|_{E}=\left(\int_{R^{3}}\big(|(-\Delta)^{\frac{s}{2}}h(1)|^{2}+V(x)h^{2}(1)\big)dx\right)^{\frac{1}{2}}>\rho \  \mbox{and} \
I_{\lambda}(h(1))<0.
\end{align*}
Choose $e=h(1)$, we obtain the conclusion.
\end{proof}
\end{lemma}
\section{Proof of Theorem \ref{theorem}}
\begin{lemma}\label{first}
Assume that \emph{(V)}, $(A1)$, $(A3)$ hold. Then $I_{\lambda}$ satisfies the Palais-Smale condition on $E$ for fixed $\lambda\in (0,1]$.
\end{lemma}
\begin{proof}
Let $\{u_{n}\}$ be a Palais-Smale sequence in $E$, i.e., $I_{\lambda}(u_{n})$ is bounded and $I'_{\lambda}(u_{n})\rightarrow 0$. We will show that $\{u_{n}\}$
has a convergent subsequence in $E$.
Then
\begin{align}\label{ps}
C+\|u_{n}\|_{E}&\geq I_{\lambda}(u_{n})-\frac{1}{\mu}\langle I'_{\lambda}(u_{n}),u_{n}\rangle\nonumber\\
&=\left(\frac{1}{2}-\frac{1}{\mu}\right)\|u_{n}\|^{2}_{H^{s}_{\lambda}}+\left(\frac{1}{2}-\frac{1}{\mu}\right)\|u_{n}\|^{2}_{2}+\left(\frac{1}{4}-\frac{1}{\mu}\right)\int_{R^{3}}\phi^{t}_{u_{n}}u^{2}_{n}dx\nonumber\\
&\quad+\int_{R^{3}}\left(\frac{u_{n}f(x,u_{n})}{\mu}-F(x,u_{n})\right)dx\nonumber\\
&\geq \left(\frac{1}{2}-\frac{1}{\mu}\right)\lambda\|u_{n}\|^{2}_{E}.
\end{align}
This implies that $\{u_{n}\}$ is bounded in $E$. \\
Up to a subsequence, we assume that $u_{n}\rightharpoonup u$ in $E$. Since $E$ is compactly embedded in $L^{p}(\mathbb{R}^{3})$ for $2\leq p<2^{\ast}_{s}$, then $u_{n}\rightarrow u$ in $L^{p}(\mathbb{R}^{3})$, $2\leq p<2^{\ast}_{s}$. By (\ref{langle}), (\ref{functional}), we get
\begin{align}\label{norm}
\|u_{n}-u\|^{2}_{\lambda}&=\langle I'_{\lambda}(u_{n})-I'_{\lambda}(u), u_{n}-u\rangle-\|u_{n}-u\|^{2}_{2}-\int_{\mathbb{R}^{3}}(\phi^{t}_{u_{n}}u_{n}-\phi^{t}_{u}u)(u_{n}-u)dx\nonumber\\
&\quad+\int_{\mathbb{R}^{3}}(f(x,u_{n})-f(x,u))(u_{n}-u)dx.
\end{align}
Clearly, we have
\begin{align}\label{have}
\langle I'_{\lambda}(u_{n})-I'_{\lambda}(u), u_{n}-u\rangle \rightarrow0 \ \ \mbox{and} \ \|u_{n}-u\|^{2}_{2}\rightarrow 0  \ \mbox{as}\  n\rightarrow \infty.
\end{align}
By the generalization of H$\ddot{\mbox{o}}$lder inequality, Lemma \ref{best} and (\ref{phi}), it follows that
\begin{align*}
\left|\int_{\mathbb{R}^{3}}\phi^{t}_{u_{n}}u_{n}(u_{n}-u)dx\right|&\leq \|\phi^{t}_{u_{n}}\|_{2^{\ast}_{t}}\|u_{n}\|_{\frac{12}{3+2t}}\|u_{n}-u\|_{\frac{12}{3+2t}}\\
&\leq C\|\phi^{t}_{u_{n}}\|_{D^{t,2}}\|u_{n}\|_{\frac{12}{3+2t}}\|u_{n}-u\|_{\frac{12}{3+2t}}\\
&\leq C\|u_{n}\|^{3}_{\frac{12}{3+2t}}\|u_{n}-u\|_{\frac{12}{3+2t}}\\
&\leq C\|u_{n}\|^{3}_{E}\|u_{n}-u\|_{\frac{12}{3+2t}}.
\end{align*}
Similarly,
\begin{align*}
\left|\int_{\mathbb{R}^{3}}\phi^{t}_{u}u(u_{n}-u)dx\right|\leq C\|u\|^{3}_{E}\|u_{n}-u\|_{\frac{12}{3+2t}}.
\end{align*}
We have
\begin{align}\label{put}
&\left|\int_{\mathbb{R}^{3}}(\phi^{t}_{u_{n}}u_{n}-\phi^{t}_{u}u)(u_{n}-u)dx\right|\leq \left|\int_{\mathbb{R}^{3}}\phi^{t}_{u_{n}}u_{n}(u_{n}-u)dx\right|+\left|\int_{\mathbb{R}^{3}}\phi^{t}_{u}u(u_{n}-u)dx\right|\rightarrow 0 \ \mbox{as}\ n\rightarrow \infty.
\end{align}
By $(A1)$, H$\ddot{\mbox{o}}$lder inequality and Minkowski inequality,
\begin{align}\label{f}
&\left|\int_{\mathbb{R}^{3}}(f(x,u_{n})-f(x,u))(u_{n}-u)dx\right|\nonumber\\
&\leq C_{1}\int_{\mathbb{R}^{3}}(|u_{n}|+|u|)|u_{n}-u|dx+C_{1}\int_{\mathbb{R}^{3}}(|u_{n}|^{p-1}+|u|^{p-1})|u_{n}-u|dx\nonumber\\
&\leq C_{1}\||u_{n}|+|u|\|_{2}\|u_{n}-u\|_{2}+C_{1}\||u_{n}|^{p-1}+|u|^{p-1}\|_{\frac{p}{p-1}}\|u_{n}-u\|_{p}\nonumber\\
&\leq C_{1}(\|u_{n}\|_{2}+\|u\|_{2})\|u_{n}-u\|_{2}+C_{1}(\|u_{n}\|^{p-1}_{p}+\|u\|^{p-1}_{p})\|u_{n}-u\|_{p}\nonumber\\
 &\leq C(\|u_{n}\|_{E}+\|u\|_{E})\|u_{n}-u\|_{2}+C(\|u_{n}\|^{p-1}_{E}+\|u\|^{p-1}_{E})\|u_{n}-u\|_{p}\rightarrow 0 \ \mbox{as} \ n\rightarrow \infty.
\end{align}
By (\ref{norm}), (\ref{have}), (\ref{put}) and (\ref{f}), we see that $\{u_{n}\}$ converges strongly in $E$ for fixed $\lambda\in (0,1]$, therefore $I_{\lambda}$ satisfies the Palais-Smale condition on $E$ for fixed $\lambda\in (0,1]$.
\end{proof}

\begin{theorem}\label{sequence}
Assume that $(A3)$ hold. Let $\lambda_{n}\rightarrow 0$ and let $\{u_{n}\}\subset E$  be a sequence of critical points of $I_{\lambda_{n}}$ satisfying $I'_{\lambda_{n}}(u_{n})=0$ and
$I_{\lambda_{n}}(u_{n})\leq C$ for some $C$ independent of $n$. Then up to a subsequence as $n\rightarrow \infty$, $u_{n}\rightharpoonup u$ in $H^{s}(\mathbb{R}^{3})$,  $u$ is a critical point of $I$.
\end{theorem}
\begin{proof}
By $I'_{\lambda_{n}}(u_{n})=0$ and $I_{\lambda_{n}}(u_{n})\leq C$,  we have
\begin{align}\label{potential}
C&\geq I_{\lambda_{n}}(u_{n})-\frac{1}{\mu}\langle I'_{\lambda_{n}}(u_{n}), u_{n}\rangle\nonumber\\
&=\left(\frac{1}{2}-\frac{1}{\mu}\right)\|u_{n}\|^{2}_{H^{s}_{\lambda}}+\left(\frac{1}{2}-\frac{1}{\mu}\right)\|u_{n}\|^{2}_{2}+\left(\frac{1}{4}-\frac{1}{\mu}\right)\int_{R^{3}}\phi^{t}_{u_{n}}u^{2}_{n}dx\nonumber\\
&\quad+\int_{R^{3}}\left(\frac{u_{n}f(x,u_{n})}{\mu}-F(x,u_{n})\right)dx\nonumber\\
&\geq \left(\frac{1}{2}-\frac{1}{\mu}\right)\|u_{n}\|^{2}_{H^{s}}+\left(\frac{1}{2}-\frac{1}{\mu}\right)\int_{R^{3}}\lambda_{n}V(x)u^{2}_{n}dx.
\end{align}
Then up to a subsequence, we have $u_{n}\rightharpoonup u$ in $H^{s}(\mathbb{R}^{3})$. By Lemma 2.3 in \cite{FSP},  $\phi^{t}_{u_{n}}\rightharpoonup\phi^{t}_{u}$ in $D^{t,2}(\mathbb{R}^{3})$. Taking $v\in C^{\infty}_{0}(\mathbb{R}^{3})$, then
\begin{align*}
\int_{\mathbb{R}^{3}}\phi^{t}_{u_{n}}uvdx\rightarrow \int_{\mathbb{R}^{3}}\phi^{t}_{u}uvdx, \ \mbox{as} \ n\rightarrow \infty.
\end{align*}
From the generalization of H$\ddot{\mbox{o}}$lder inequality, it follows that
\begin{align*}
\left|\int_{\mathbb{R}^{3}}\phi^{t}_{u_{n}}(u_{n}-u)vdx\right|&\leq \|\phi^{t}_{u_{n}}\|_{2^{\ast}_{t}}\|u_{n}-u\|_{L^{12/(3+2t)}(\Omega)}\|v\|_{L^{12/(3+2t)}(\Omega)}\rightarrow 0 \ \mbox{as} \ n\rightarrow \infty,
\end{align*}
where $\Omega$ is the support of $v$. Then,
\begin{align*}
\left|\int_{\mathbb{R}^{3}}\phi^{t}_{u_{n}}u_{n}vdx-\int_{\mathbb{R}^{3}}\phi^{t}_{u}uvdx\right|\leq \left|\int_{\mathbb{R}^{3}}(\phi^{t}_{u_{n}}-\phi^{t}_{u})uvdx\right|+\left|\int_{\mathbb{R}^{3}}\phi^{t}_{u_{n}}(u_{n}-u)vdx\right|\rightarrow 0
\end{align*}
as $n\rightarrow \infty$, for all $v\in C^{\infty}_{0}(\mathbb{R}^{3})$. \\
By (\ref{langle}),  (\ref{functional}),
\begin{align*}
\langle I'_{\lambda_{n}}(u_{n}),v\rangle=\int_{\mathbb{R}^{3}}\left((-\Delta)^{\frac{s}{2}}u_{n}(-\Delta)^{\frac{s}{2}}v+u_{n}v+\phi^{t}_{u_{n}}u_{n}v-f(x,u_{n})v\right)dx+\lambda_{n}\int_{\mathbb{R}^{3}}V(x)u_{n}vdx,
\end{align*}
where $v\in C^{\infty}_{0}(\mathbb{R}^{3})$. By (\ref{potential}), H$\ddot{\mbox{o}}$lder inequality,
\begin{align*}
\lambda_{n}\int_{\mathbb{R}^{3}}V(x)u_{n}v=\lambda_{n}\int_{\mathbb{R}^{3}}(\sqrt{V(x)}u_{n})(\sqrt{V(x)}v)dx\leq \lambda^{1/2}_{n}\left(\int_{\mathbb{R}^{3}}\lambda_{n}V(x)v^{2}dx\right)^{1/2}\left(\int_{\mathbb{R}^{3}}V(x)v^{2}dx\right)^{1/2}
\end{align*}
$\rightarrow 0$ as $n\rightarrow \infty$. Thus, we see that $I'(u)v=0$ for all $v\in C^{\infty}_{0}(\mathbb{R}^{3})$. It is known that $C^{\infty}_{0}(\mathbb{R}^{3})$ is dense in $H^{s}(\mathbb{R}^{3})$, see Theorem 2.4 in \cite{DPV}. Therefore, $I'(u)v= 0$ for all $v\in H^{s}(\mathbb{R}^{n})$, $u$ is a critical point of $I$.
\end{proof}

\begin{lemma}\label{fractional}\emph{(Lemma 2.3 in \cite{Avenia}).}
Let $B_{\sigma}(x)$  be the open ball in $\mathbb{R}^{3}$ of radius $\sigma$ centred at $x$. If $\{u_{n}\}$ is bounded in $H^{s}(\mathbb{R}^{3})$ and for $q\in [2,2^{\ast}_{s})$, we have
\begin{align}
\sup_{x\in \mathbb{R}^{3}}\int_{B_{\sigma}(x)}|u_{n}|^{q}\rightarrow 0 \ \mbox{as} \ n\rightarrow \infty,
\end{align}
then $u_{n}\rightarrow 0$ in $L^{r}(\mathbb{R}^{3})$ for $r\in (2,2^{\ast}_{s})$.
\end{lemma}
\textbf{Proof of Theorem 1.}
Choose $\phi\in C^{\infty}_{0}(\mathbb{R}^{3})$ and $\xi>0$. Define a path $h:[0,1]\rightarrow E$ by $h(\eta)=\eta\xi\phi$. When $\xi$ is large enough, by lemma \ref{lambda}, we have $I_{\lambda}(h(1))<0$, $\|h(1)\|_{E}>\rho$ for small $\rho$ and $\sup_{\eta\in [0,1]}I_{\lambda}(\gamma(\eta))\leq c$ for some $c$ independent of $\eta\in [0,1]$. Define
\begin{align*}
c_{\lambda}=\inf_{\gamma\in \Gamma}\sup_{\eta\in [0,1]}I_{\lambda}(\gamma(\eta)),
\end{align*}
where $\Gamma=\{\gamma|\gamma\in C([0,1],E), \gamma(0)=0, \gamma(1)=\xi\phi\}$. By Lemma \ref{geometry}, the mountain pass theorem holds and $c_{\lambda}$ is a critical value of $I_{\lambda}$. Therefore, we can choose $\lambda_{n}\rightarrow 0$, and a sequence of critical points  $\{u_{n}\}\subset E$  satisfying $I'_{\lambda_{n}}(u_{n})=0$ and $I_{\lambda_{n}}(u_{n})\leq c$.
By Theorem \ref{sequence}, up to a subsequence $u_{n}\hookrightarrow u$, and $u$ is a critical point of $I$ in $H^{s}(\mathbb{R}^{3})$. We need to show that $u\neq0$. Note that $H^{s}(\mathbb{R}^{3})\hookrightarrow L^{q}(\mathbb{R}^{3})$ for any $q\in [2,2^{\ast}_{s}]$, then
\begin{align*}
0=I'_{\lambda_{n}}(u_{n})u_{n}&=\int_{\mathbb{R}^{3}}(|(-\Delta)^{\frac{s}{2}}u_{n}|^{2}+u^{2}_{n})dx+\int_{\mathbb{R}^{3}}\phi^{t}_{u_{n}}(x)u^{2}_{n}dx-\int_{\mathbb{R}^{3}}f(x,u_{n})dx\\
&\quad+\lambda_{n}\int_{\mathbb{R}^{3}}V(x)u^{2}_{n}dx\\
&\geq \|u_{n}\|^{2}_{H^{s}}-\varepsilon\|u_{n}\|^{2}_{2}-C_{\varepsilon}\|u_{n}\|^{p}_{p}\\
&\geq C\|u_{n}\|^{2}_{p}-C_{\varepsilon}\|u_{n}\|^{p}_{p}.
\end{align*}
Then $\|u_{n}\|_{p}\geq (\frac{C}{C_{\varepsilon}})^{\frac{1}{p-2}}$. If $u=0$, then given any $x\in \mathbb{R}^{3}$,  $u_{n}\rightharpoonup 0$ in $H^{s}(B_{\sigma}(x))$, since the embedding $H^{s}(\mathbb{R}^{3})\hookrightarrow L^{q}(\mathbb{R}^{3})$
 is locally compact for $q\in [1, 2^{\ast}_{s})$, then $u_{n}\rightarrow 0$ in $L^{p}(B_{\sigma}(x))$. By Lemma \ref{fractional}, we have $u_{n}\rightarrow 0$ in $L^{p}(\mathbb{R}^{3})$, which is a contradiction with $\|u_{n}\|_{p}\geq (\frac{C}{C_{\varepsilon}})^{\frac{1}{p-2}}$.
 Therefore, $u\neq 0$. The proof is complete.


\end{document}